\documentclass{amsart}
\usepackage{amssymb, amsmath, mathrsfs, verbatim, url, bbm}

\makeatletter
\@namedef{subjclassname@2020}{%
  \textup{2020} Mathematics Subject Classification}
\makeatother

\theoremstyle{plain}
\newtheorem{theorem}{Theorem}[section]
\newtheorem{lemma}[theorem]{Lemma}
\newtheorem{proposition}[theorem]{Proposition}
\newtheorem{corollary}[theorem]{Corollary}

\theoremstyle{definition}
\newtheorem{definition}[theorem]{Definition}

\newcommand{\Gr}{\mathsf{Gr}}

\newcommand{\bS}{\mathbf{\Sigma}}
\newcommand{\bP}{\mathbf{\Pi}}
\newcommand{\bD}{\mathbf{\Delta}}
\newcommand{\inte}{\mathsf{int}}
\newcommand{\leqT}{\leq_{\mathsf{T}}}
\newcommand{\equivT}{\equiv_{\mathsf{T}}}
\newcommand{\AC}{\mathsf{AC}}
\newcommand{\AD}{\mathsf{AD}}

\newcommand{\BP}{\mathsf{BP}}

\newcommand{\HH}{\mathcal{H}}
\newcommand{\PP}{\mathcal{P}}
\newcommand{\UU}{\mathcal{U}}
\newcommand{\RRR}{\mathbb{R}}
\newcommand{\QQQ}{\mathbb{Q}}

\begin{document}

\title{On the scope of the Effros theorem}

\author{Andrea Medini}
\address{Institut f\"{u}r Diskrete Mathematik und Geometrie
\newline\indent Technische Universit\"at Wien
\newline\indent  Wiedner Hauptstra\ss e 8–-10/104
\newline\indent 1040 Vienna, Austria}
\email{andrea.medini@tuwien.ac.at}
\urladdr{http://www.dmg.tuwien.ac.at/medini/}

\date{March 12, 2022}

\begin{abstract}
All spaces (and groups) are assumed to be separable and metrizable. Jan van Mill showed that every analytic group $G$ is Effros (that is, every continuous transitive action of $G$ on a non-meager space is micro-transitive). We complete the picture by obtaining the following results:
\begin{itemize}
\item Under $\AC$, there exists a non-Effros group,
\item Under $\AD$, every group is Effros,
\item Under $\mathsf{V=L}$, there exists a coanalytic non-Effros group.
\end{itemize}
The above counterexamples will be graphs of discontinuous homomorphisms.
\end{abstract}



\maketitle

\tableofcontents

\section{Introduction}\label{section_introduction}

Throughout this article, we will be working in the theory $\mathsf{ZF}+\mathsf{DC}$, that is, the usual axioms of Zermelo-Fraenkel (without the Axiom of Choice, which we will denote by $\AC$) plus the principle of Dependent Choices, and we will denote by $\AD$ the Axiom of Determinacy (see \cite[Section 2]{carroy_medini_muller} for a thorough discussion). All spaces (including groups) will be assumed to be separable and metrizable. Furthermore, by group we will mean topological group, and by group action, we will mean continuous group action (see Section \ref{section_preliminaries} for more details).

The research presented here ultimately stems from the seminal work of E. G. Effros from \cite{effros}, which was in turn inspired by results of J. Glimm from \cite{glimm}. The results of \cite{effros} have impacted three rather diverse fields: $C^\ast$-algebras (this was the original motivation for Glimm and Effros), descriptive set theory (see \cite[Sections 3 and 4]{kechris_1994}), and topology (see \cite[Section~1]{charatonik_mackowiak} and \cite[Section 1]{van_mill_2008}). In fact, in his MathSciNet review of \cite{van_mill_2008}, V. Pestov described Theorem \ref{theorem_effros} as ``arguably the most important single result concerning Polish (= separable completely metrizable) topological groups.'' Furthermore, as is well-known, the classical Open Mapping Theorem and Closed Graph Theorem for separable Banach spaces easily follow from it (see Section \ref{section_classics} for much more general statements).

Here, we will follow the topologically-minded approach of \cite{ancel}, in which F. D. Ancel presented an alternative version of the results of Effros. In particular, he introduced to following notion.
\begin{definition}[Ancel]
Let $G$ be a group acting on a space $X$. We will say that this action is \emph{micro-transitive} if $Ux$ is a neighborhood of $x$ for every $x\in X$ and every neighborhood $U$ of the identity in $G$.
\end{definition}

Recall that an action of a group $G$ on a space $X$ is \emph{transitive} if for every $x,y\in X$ there exists $g\in G$ such that $gx=y$. The following will be the crucial notion for the remainder of the paper.
\begin{definition}
Let $G$ be a group. We will say that $G$ is \emph{Effros} if every transitive action of $G$ on a non-meager space is micro-transitive.
\end{definition}

At this point, the Effros theorem admits a particularly concise formulation (see \cite[Theorem 1]{ancel}).
\begin{theorem}[Effros]\label{theorem_effros}
Every Polish group is Effros.
\end{theorem}

The above theorem was substantially generalized by J. van Mill, who obtained the following result (see \cite[Theorem 1.2]{van_mill_2004}).

\begin{theorem}[van Mill]\label{theorem_van_mill}
Every analytic group is Effros.
\end{theorem}

It seems natural to wonder whether the above result is optimal. We will show that this is indeed the case. More precisely, we will show that the property of being an Effros group behaves like the classical regularity properties. In fact, Corollary \ref{corollary_main_PB} shows that every group is Effros under~$\AD$, Corollary \ref{corollary_counterexample_ZFC} shows that this is not the case in $\mathsf{ZFC}$, and Corollary \ref{corollary_counterexample_V=L} exhibits a coanalytic non-Effros group under $\mathsf{V=L}$.

Finally, we remark that generalizations of Theorem \ref{theorem_van_mill} to the non-separable realm do exist (see \cite{ostaszewski_2013} and \cite{ostaszewski_2015}), although the situation is not as pleasant as in the separable case.

\section{Preliminaries and terminology}\label{section_preliminaries}

Given a group $G$ with identity $e$ and a set $X$, a function $\cdot:G\times X\longrightarrow X$ is called a \emph{group action} if the following conditions hold:
\begin{itemize}
\item $e\cdot x=x$ for every $x\in X$,
\item $(gh)\cdot x=g\cdot(h\cdot x)$ for every $g,h\in G$ and $x\in X$.
\end{itemize}
We will often simply write $gx$ instead of $g\cdot x$. Given $S\subseteq G$ and $x\in X$, we will use the notation $Sx=\{gx:g\in S\}$.

For simplicity, we will always assume that every group $G$ is also endowed with a topology which makes it into a topological group, and that every set $X$ on which $G$ acts is a topological space. As we mentioned at the very beginning of the article, all groups and spaces will be assumed to be separable and metrizable. Furthermore, all group actions will be assumed to be continuous.\footnote{\,We remark that, as in \cite{van_mill_2004}, all results contained in this article actually hold for separately continuous actions (that is, those actions such that all functions $x\longmapsto gx$ and $g\longmapsto gx$ are continuous).}

Notice that, for every given $g\in G$, the function $x\longmapsto gx$ is a homeomorphism of $X$ (with inverse $x\longmapsto g^{-1}x$). Given $x\in X$, define $\gamma_x:G\longrightarrow X$ by setting $\gamma_x(g)=gx$ for every $g\in G$. The following simple proposition (see \cite[Lemma 1]{ancel}) gives some useful characterizations of micro-transitivity.
\begin{proposition}\label{proposition_micro-transitivity}
Let $G$ be a group acting transitively on a space $X$. Then the following conditions are equivalent:
\begin{itemize}
\item $G$ acts micro-transitively on $X$,
\item $\gamma_x$ is open for every $x\in X$,	
\item $\gamma_x$ is open for some $x\in X$.
\end{itemize}	
\end{proposition}

Our reference for descriptive set theory is \cite{kechris_1995}. In particular, we assume familiarity with the basic theory of Polish spaces, and their Borel and projective subsets. Our reference for other set-theoretic notions is \cite{jech}. Given a set $X$, we will denote by $\PP(X)$ the collection of all subsets of $X$. We will denote by $X^{\leq\omega}$ the collection of all functions $s:n\longrightarrow X$, where $n\leq\omega$. Given a function $f:X\longrightarrow Y$, we will denote by
$$
\Gr(f)=\{(x,f(x)):x\in X\}
$$
the graph of $f$. While $\Gr(f)=f$ from a purely set-theoretic standpoint, we believe that this notation will improve the readability of the article. Observe that, when $G$ and $H$ are groups, and $\varphi:G\longrightarrow H$ is a homomorphism, then $\Gr(\varphi)$ has a natural group structure (in fact, it is a subgroup of $G\times H$).

Given spaces $X$ and $Y$, we will say that $j:X\longrightarrow Y$ is an \emph{embedding} if $j:X\longrightarrow j[X]$ is a homeomorphism. Given $1\leq n<\omega$, we will say that a space $X$ is $\bS^1_n$ (respectively $\bP^1_n$ or $\bD^1_n$) if there exists a Polish space $Z$ and an embedding $j:X\longrightarrow Z$ such that $j[X]\in\bS^1_n(Z)$ (respectively $j[X]\in\bP^1_n(Z)$ or $j[X]\in\bD^1_n(Z)$). It is easy to show that a space $X$ is $\bS^1_n$ (respectively $\bP^1_n$ or $\bD^1_n$) iff $j[X]\in\bS^1_n(Z)$ (respectively $j[X]\in\bP^1_n(Z)$ or $j[X]\in\bD^1_n(Z)$) for every Polish space $Z$ and every embedding $j:X\longrightarrow Z$ (see \cite[Proposition 4.2]{medini_zdomskyy}). A space is \emph{analytic} (respectively \emph{coanalytic} or \emph{Borel}) if it is $\bS^1_1$ (respectively $\bP^1_1$ or $\bD^1_1$).

We conclude with some preliminaries concerning the Baire property. Given a space $X$, we will denote by $\BP(X)$ be the collection of all subsets of $X$ that have the Baire property in $X$. We will denote by $\BP$ the statement that $\BP(\omega^\omega)=\PP(\omega^\omega)$. It is well-known that $\AD$ implies $\BP$ (use the methods of \cite[Section 8.H]{kechris_1995}). In Section \ref{section_positive}, we will need the following simple consequence of $\BP$.

\begin{proposition}\label{proposition_PB}
Assume $\BP$. Then $\BP(X)=\PP(X)$ for every space $X$.
\end{proposition}

\begin{proof}
By \cite[Exercise 7.14]{kechris_1995}, for every non-empty Polish space $X$ there exists a continuous open surjection $f:\omega^\omega\longrightarrow X$. Using this fact, it is easy to see that the desired result holds when $X$ is Polish. The general case follows from the fact that for every space $X$ there exists a Polish space $Z$ and an embedding $j:X\longrightarrow Z$ such that $j[X]$ is dense in $Z$, since $\BP(Z)\cap\PP(j[X])\subseteq\BP(j[X])$ by density.
\end{proof}

\section{Revisiting some classics and non-Effros groups in $\mathsf{ZFC}$}\label{section_classics}

We begin by showing how the Effros theorem easily implies two classical theorems. More precisely, in the separable context, Corollary \ref{corollary_closed_graph} considerably strengthens the Closed Graph Theorem, while Corollary \ref{corollary_open_mapping} considerably strengthens the Open Mapping Theorem. In particular, they show that the linear structure is irrelevant in this context. While these facts seem to be somewhat folklore, the concept of an Effros group allows for especially general and elegant statements.

While Corollary \ref{corollary_closed_graph}, Theorem \ref{theorem_open_mapping} and Corollary \ref{corollary_open_mapping} seem to be of independent interest, they will not play much of a role in the remainder of the paper. On the other hand, Theorem \ref{theorem_closed_graph} is the crucial tool that will allow us to obtain examples of non-Effros groups (see Corollaries \ref{corollary_counterexample_ZFC} and \ref{corollary_counterexample_V=L}).

\begin{theorem}\label{theorem_closed_graph}
Let $G$ and $H$ be groups, and let $\varphi:G\longrightarrow H$ be a homomorphism. If $\Gr(\varphi)$ is an Effros group and $G$ is non-meager then $\varphi$ is continuous.
\end{theorem}
\begin{proof}
Assume that $\Gr(\varphi)$ is an Effros group and that $G$ is non-meager. Consider the action $\cdot$ of $\Gr(\varphi)$ on $G$ obtained by setting
$$
(g,\varphi(g))\cdot x=gx
$$
for every $g,x\in G$. Obviously, the action $\cdot$ is transitive. Since $\Gr(\varphi)$ is Effros and $G$ is non-meager, it follows that $\cdot$ is micro-transitive. Therefore, by Proposition \ref{proposition_micro-transitivity}, the bijection $\gamma_e:\Gr(\varphi)\longrightarrow G$ associated to $\cdot$ is open, where $e$ denotes the identity of $G$. This means that $\gamma_e^{-1}$ is continuous, hence so is $\varphi=\pi\circ\gamma_e^{-1}$, where $\pi:G\times H\longrightarrow H$ denotes the natural projection.
\end{proof}

In the proof of the following corollary, for the sake of concreteness, we will describe a specific discontinuous group homomorphism. For several other suitable examples, see \cite[Section 1]{rosendal}.

\begin{corollary}\label{corollary_counterexample_ZFC}
Assume $\AC$. Then there exists a non-Effros group.\footnote{\,It is easy to verify that the example given here is meager. We remark that every non-principal ultrafilter $\UU$ on $\omega$ with its natural group structure (see \cite[Section 3]{medini}) gives a Baire example of non-Effros group. To see this, observe that the characteristic function $\chi:2^\omega\longrightarrow 2$ of $\UU$ is a discontinuous group homomorphism, and that $\Gr(\chi)$ is isomorphic to $\UU\times 2$ as a topological group. Finally, using the methods of \cite[Proposition 13.6]{medini}, one can show that $\UU\times 2$ is isomorphic to $\UU$ as a topological group.}
\end{corollary}
\begin{proof}
Using $\AC$, we can fix a basis $\HH$ for $\RRR$ as a vector space over $\QQQ$. Since $\HH$ is uncountable, we can pick $h_\infty\in\HH$ and $h_n\in\HH\setminus\{h_\infty\}$ for $n\in\omega$ such that $h_n\to h_\infty$. Let $\varphi:\RRR\longrightarrow\QQQ$ be the unique linear functional such that
$$
\left.
\begin{array}{lcl}
& & \varphi(h)= \left\{
\begin{array}{ll}
1 & \text{if }h=h_\infty,\\
0 & \text{if }h\neq h_\infty
\end{array}
\right.
\end{array}
\right.
$$
for every $h\in\HH$. Notice that $\varphi$ is discontinuous as $\varphi(h_n)=0\not\to 1=\varphi(h_\infty)$. It follows from Theorem \ref{theorem_closed_graph} that $\Gr(\varphi)$ is a non-Effros group.
\end{proof}

\begin{corollary}\label{corollary_closed_graph}
Let $G$ and $H$ be groups, and let $\varphi:G\longrightarrow H$ be a homomorphism. If $\Gr(\varphi)$ is analytic and $G$ is non-meager then $\varphi$ is continuous.\footnote{\,This corollary can also be derived from \cite[Theorem 9.10]{kechris_1995}. The same remark holds for Corollary~\ref{corollary_closed_graph_PB}.}
\end{corollary}
\begin{proof}
Combine Theorems \ref{theorem_closed_graph} and \ref{theorem_van_mill}.
\end{proof}

\begin{theorem}\label{theorem_open_mapping}
Let $G$ and $H$ be groups, and let $\varphi:G\longrightarrow H$ be a surjective continuous homomorphism. If $G$ is Effros and $H$ is non-meager then $\varphi$ is open.	
\end{theorem}
\begin{proof}
Consider the action $\cdot$ of $G$ on $H$ defined by setting $g\cdot h=\varphi(g)h$, and notice that $\cdot$ is transitive because $\varphi$ is surjective. Since $G$ is Effros and $H$ is non-meager, it follows that $\cdot$ is micro-transitive. Therefore $\gamma_e$ is open by Proposition \ref{proposition_micro-transitivity}, where $e$ denotes the identity of $H$. The proof is concluded by observing that $\gamma_e=\varphi$.
\end{proof}

\begin{corollary}\label{corollary_open_mapping}
Let $G$ and $H$ be groups, and let $\varphi:G\longrightarrow H$ be a surjective continuous homomorphism. If $G$ is analytic and $H$ is non-meager then $\varphi$ is open.	
\end{corollary}
\begin{proof}
Combine Theorems \ref{theorem_open_mapping} and \ref{theorem_van_mill}.
\end{proof}

We remark that the assumption ``non-meager'' cannot be dropped in any of Theorem \ref{theorem_closed_graph}, Corollary \ref{corollary_closed_graph}, Theorem \ref{theorem_open_mapping} or Corollary \ref{corollary_open_mapping}. To see this, let $\QQQ_d$ denote the rational numbers with the discrete topology, and consider the identity function $\varphi:\QQQ\longrightarrow\QQQ_d$. It is easy to realize that $\varphi$ gives the desired counterexample in the case of the first two results, while $\varphi^{-1}$ gives the desired counterexample in the case of the last two. These examples are inspired by \cite[Remark (4)]{van_mill_2004}.

\section{Positive results}\label{section_positive}

The main result of this section is Corollary \ref{corollary_main_PB}, which shows that, when the set-theoretic universe is sufficiently regular, no assumption at all on the complexity of the group is needed in Theorems \ref{theorem_effros} and \ref{theorem_van_mill}. This will follow from Theorem \ref{theorem_main_positive}, whose proof is a minor modification of van Mill's proof of Theorem \ref{theorem_van_mill} from \cite{van_mill_2004}.

In the earlier part of the paper we preferred to cite $\AD$, since we regard it as a more ``quotable'' axiom. However, here we will be more precise, and point out that $\BP$ (that is, the assumption that all sets of reals have the Baire property) is in fact sufficient to obtain all the results that we are interested in (see Corollaries \ref{corollary_main_PB}, \ref{corollary_closed_graph_PB}, and \ref{corollary_open_mapping_PB}).

Before stating the following preliminary results, we clarify some terminology, as ours differs from van Mill's. We will say that a subset $A$ of a space $Z$ is \emph{nowhere meager} if $A\cap U$ is non-meager in $Z$ for every non-empty open subset $U$ of $Z$.\footnote{\,Such sets are called \emph{fat} by van Mill (in fact, Corollary \ref{corollary_fat} is the analogue of \cite[Proposition 2.2]{van_mill_2004} in our context). On the other hand, he says that $A$ is \emph{nowhere meager} if every non-empty open subset of $A$ is non-meager in $Z$. It seems almost blasphemous to disagree with van Mill, but we find our choice of terminology more natural.}

\begin{proposition}\label{proposition_fat}
Let $Z$ be a space and let $A\in\BP(Z)$. If $A$ is nowhere meager then $A$ is comeager.	
\end{proposition}
\begin{proof}
This is simply the dual version of \cite[Proposition 8.26]{kechris_1995}.
\end{proof}

\begin{corollary}\label{corollary_fat}
Let $Z$ be a non-meager space and let $A,B\in\BP(Z)$. If $A$ and $B$ are nowhere meager then $A\cap B\neq\varnothing$.	
\end{corollary}

\begin{theorem}\label{theorem_main_positive}
Let $G$ be a group that acts transitively on a non-meager space $X$. Assume that $Ux\in\BP(X)$ for every open subset $U$ of $G$ and every $x\in X$. Then the action is micro-transitive.
\end{theorem}
\begin{proof}
Let $e$ be the identity of $G$. As in \cite[Section 3]{van_mill_2004}, fix open neighborhoods $U_n$ of $e$ for $n\in\omega$ satisfying the following conditions for each $n$:
\begin{itemize}
\item $U_n^{-1}=U_n$,
\item $U_{n+1}\subseteq U_{n+1}^2\subseteq U_n$.
\end{itemize}

The following two claims correspond to \cite[Corollary 3.3 and Lemma 3.4]{van_mill_2004} respectively, and can be proved in the same way. Given $S\subseteq X$, we will denote by $\overline{S}$ the closure of $S$ in $X$, and we will denote by $\inte(S)$ the interior of $S$ in $X$.

\noindent\textbf{Claim 1.} Let $x\in X$ and $n\in\omega$. If $V$ is open in $X$ and $V\cap U_n x\neq\varnothing$ then $V\cap U_n x$ is non-meager in $X$.

\noindent\textbf{Claim 2.} Let $x\in X$ and $n\in\omega$. Then $x\in \inte(\overline{U_{n+1}x})$ and $\inte(\overline{U_{n+1}x})$ is dense in $\overline{U_{n+1}x}$.

\noindent\textbf{Claim 3.} Let $x\in X$ and $n\in\omega$. Then $\inte(\overline{U_{n+1}x})\subseteq U_n x$.

\noindent\textit{Proof of Claim 3.} Set $V=\inte(\overline{U_{n+1}x})$, and pick $z\in V$. Set $W=\inte(\overline{U_{n+1}z})$. Set $Z=V\cap W$, and observe that $Z$ is an open neighborhood of $z$ by Claim~2. In particular, $Z$ is non-meager, otherwise one could use the transitive action of $G$ on $X$ to contradict the assumption that $X$ is non-meager. Set $A=Z\cap U_{n+1}x$ and $B=Z\cap U_{n+1}z$. Observe that $A$ is dense in $Z$ because $Z\subseteq\overline{U_{n+1}x}$. Similarly, one sees that $B$ is dense in $Z$. Furthermore, since $U_{n+1}x,U_{n+1}z\in\BP(X)$ by assumption and $Z$ is open in $X$, it is easy to check that $A,B\in\BP(Z)$.

Pick a non-empty open subset $U$ of $Z$. Since $U$ is also open in $X$ and $A$ is dense in $Z$, Claim 1 shows that $U\cap A=U\cap U_{n+1}x$ is non-meager in~$X$, hence non-meager in $Z$. Similarly, one sees that $U\cap B=U\cap U_{n+1}z$ is non-meager in $Z$. In conclusion, both $A$ and $B$ are nowhere meager in~$Z$. It follows from Corollary \ref{corollary_fat} that $A\cap B\neq\varnothing$, so it is possible to pick $y\in A\cap B$. This means that there exist $g,h\in U_{n+1}$ such that $gx=y=hz$. Therefore
$$
z=h^{-1}gx\in U_{n+1}^{-1}U_{n+1}x=U_{n+1}^2x\subseteq U_n x,
$$
which concludes the proof of the claim. $\blacksquare$

Finally, in order to show that the action is micro-transitive, pick $x\in X$ and an open neighborhood $U$ of $e$. Fix $n\in\omega$ large enough so that $U_n\subseteq U$. Using Claims 2 and 3, one sees that
$$
x\in\inte(\overline{U_{n+1}x})\subseteq U_n x\subseteq Ux,
$$
which shows that $Ux$ is a neighborhood of $x$.
\end{proof}

\begin{corollary}\label{corollary_main_PB}
Assume $\BP$. Then every group is Effros.
\end{corollary}
\begin{proof}
Combine Theorem \ref{theorem_main_positive} and Proposition \ref{proposition_PB}.
\end{proof}

\begin{corollary}\label{corollary_closed_graph_PB}
Assume $\BP$. Let $G$ and $H$ be groups, and let $\varphi:G\longrightarrow H$ be a homomorphism. If $G$ is non-meager then $\varphi$ is continuous.	
\end{corollary}
\begin{proof}
Combine Corollary \ref{corollary_main_PB} and Theorem \ref{theorem_closed_graph}.
\end{proof}

\begin{corollary}\label{corollary_open_mapping_PB}
Assume $\BP$. Let $G$ and $H$ be groups, and let $\varphi:G\longrightarrow H$ be a surjective continuous homomorphism. If $H$ is non-meager then $\varphi$ is open.	
\end{corollary}
\begin{proof}
Combine Corollary \ref{corollary_main_PB} and Theorem \ref{theorem_open_mapping}.
\end{proof}

\section{A method of Vidny\'anszky}\label{section_vidnyanszky}

In this section we will discuss a method developed by Z. Vidny\'anszky in \cite{vidnyanszky}. First we will state the original version (see Theorem \ref{theorem_zoltan}), and then deduce a ``multivariable'' generalization (see Theorem \ref{theorem_zoltan_multivariable}), which will be needed in Section \ref{section_counterexample_V=L}. This method is a ``black-box'' version of the technique that is mostly known for the applications given by A. W. Miller in \cite{miller_1989}, and has spawned many more since then.\footnote{\,The very first instance of this idea, however, seems to have appeared in a paper by Erd\H{o}s, Kunen and Mauldin (see \cite[Theorems 13, 14 and 16]{erdos_kunen_mauldin}).} The purpose of this technique is to construct coanalytic examples of certain pathological sets of reals under the assumption $\mathsf{V=L}$.

We will assume some familiarity with the basics of recursion theory (see \cite{odifreddi}) and effective descriptive set theory (see \cite{moschovakis}). On the other hand, no previous knowledge of the theory of $\mathsf{L}$ is required apart from Theorem~\ref{theorem_zoltan}. We will denote Turing reduction by $\leqT$ and Turing equivalence by $\equivT$. When $M$ is a space in which it makes sense to consider Turing reduction, we will say that $S\subseteq M$ is \emph{cofinal in the Turing degrees} if for every $x\in M$ there exists $y\in S$ such that $x\leqT y$.

As in \cite[Definition 1.2]{vidnyanszky}, given $F\subseteq M^{\leq\omega}\times B\times M$, where $M$ and $B$ are sets of size $\omega_1$, we will say that $X\subseteq M$ is \emph{compatible with $F$} if there exist enumerations $B=\{p_\alpha:\alpha<\omega_1\}$, $X=\{x_\alpha:\alpha<\omega_1\}$ and, for every $\alpha<\omega_1$, a sequence $A_\alpha\in M^{\leq\omega}$ that is an enumeration of $\{x_\beta:\beta<\alpha\}$ in type $\leq\omega$ such that $x_\alpha\in F_{(A_\alpha,p_\alpha)}$ for every $\alpha<\omega_1$. Here, given $(A,p)\in M^{\leq\omega}\times B$, we use the notation $F_{(A,p)}=\{x\in M:(A,p,x)\in F\}$. Intuitively, one should think of $A_\alpha$ as enumerating the portion of the desired set $X$ constructed before stage $\alpha$. The section $F_{(A_\alpha,p_\alpha)}$ consists of the admissible candidates to be added at stage $\alpha$, where $p_\alpha$ encodes the current condition to be satisfied. The following result first appeared as \cite[Theorem 1.3]{vidnyanszky}.

\begin{theorem}[Vidny\'anszky]\label{theorem_zoltan}
Assume $\mathsf{V=L}$. Let $M=2^\omega$, and let $B$ be an uncountable Borel space.\footnote{\,More generally, one could replace $2^\omega$ with any other uncountable Polish space with a natural notion of Turing reducibility. A similar remark holds for Theorem \ref{theorem_zoltan_multivariable}.} Assume that $F\subseteq M^{\leq\omega}\times B\times M$ is coanalytic, and that for all $(A,p)\in M^{\leq\omega}\times B$ the section $F_{(A,p)}$ is cofinal in the Turing degrees. Then there exists a coanalytic $X\subseteq M$ that is compatible with $F$.
\end{theorem}

Unfortunately, there are situations in which more than one element needs to be added at every stage. This is the case, for example, when one wants $X$ to be a group (see \cite{kastermans} or \cite{fischer_schrittesser_tornquist}) or a Hamel basis (see \cite[Theorem 9.26 and Lemma 9.27]{miller_1989}). Since Theorem \ref{theorem_zoltan} is not suited for this purpose, we have adapted it as follows. We remark that, although he did not give the general statement, the main idea of the proof of Theorem \ref{theorem_zoltan_multivariable} is also essentially due to Vidny\'anszky (see the comment that follows \cite[Definition 4.8]{vidnyanszky}).

For the remainder of this section, for simplicity, we will assume that $Z=2^\omega$. We will also assume that a recursive partition of $\omega$ into $\xi$ infinite sets is given, where $2\leq\xi\leq\omega$, so that it will be possible to identify $Z^\xi$ with $Z$ for the purposes of Turing reduction. When the space $M$ is in the form $Z^\xi$, we will say that $S\subseteq M$ is \emph{equicofinal in the Turing degrees} if for every $a\in Z$ there exists $x\in S$ such that the following conditions~are~satisfied:
\begin{itemize}
\item $a\leqT x$,
\item $x\equivT x(n)$ for every $n\in\xi$.
\end{itemize}
\begin{theorem}\label{theorem_zoltan_multivariable}
Assume $\mathsf{V=L}$. Let $M=Z^\xi$, where $Z=2^\omega$ and $2\leq\xi\leq\omega$, and let $B$ be an uncountable Borel space. Assume that $F\subseteq M^{\leq\omega}\times B\times M$ is coanalytic, and that for all $(A,p)\in M^{\leq\omega}\times B$ the section $F_{(A,p)}$ is equicofinal in the Turing degrees. Then there exists $X\subseteq M$ such that the following conditions are satisfied:
\begin{itemize}
\item $X$ is compatible with $F$,
\item $\{x(n):x\in X\text{ and }n\in\xi\}$ is coanalytic.
\end{itemize}
\end{theorem}
\begin{proof}
First define $F'\subseteq M^{\leq\omega}\times B\times M$ by declaring $(A,p,x)\in F'$ iff
$$
(A,p,x)\in F\text{ and }\forall n\in\xi\,\big(x\equivT x(n)\big).
$$
Observe that $F'$ is coanalytic, and that for every $(A,p)\in M^{\leq\omega}\times B$ the section $F'_{(A,p)}$ is cofinal in the Turing degrees. Therefore, by Theorem \ref{theorem_zoltan}, there exists a coanalytic $X\subseteq M$ that is compatible with $F'$ (hence with $F$ as well). In particular, we can fix $a\in Z$ and a $\Pi^1_1$ formula $\phi(x,y)$ such that $x\in X$ iff $\phi(x,a)$. Set
$$
X_n=\{x(n):x\in X\}
$$
for $n\in\xi$. It will be enough to show that each $X_n$ is coanalytic. So fix $n$. Define
$$
\theta(x,y,z)\text{ iff }\big(x(n)=y\text{ and }\phi(x,z)\big),
$$
and observe that $\theta$ is also $\Pi^1_1$. Finally, define
$$
\psi(y,z)\text{ iff }\exists x\in\Delta^1_1(y)\,\theta(x,y,z),
$$
and observe that $\psi$ is $\Pi^1_1$ by the Spector-Gandy Theorem (see \cite[Corollary 29.3]{miller_1995}).\footnote{\,Miller works in $\omega^\omega$, but all the relevant arguments carry over to $2^\omega$. Alternatively, one can apply the second part of \cite[Theorem 4D.3]{moschovakis} to the Spector pointclass $\Gamma=\Pi^1_1(a)$.} Using the well-known fact that $x\leqT y$ implies $x\in\Delta^1_1(y)$, it is straightforward to verify that $y\in X_n$ iff $\psi(y,a)$, which concludes the proof.
\end{proof}

\section{A counterexample under $\mathsf{V=L}$}\label{section_counterexample_V=L}

The main result of this section is Theorem \ref{theorem_counterexample_V=L}, which gives a consistent example of a discontinuous group homomorphism between Polish groups whose graph is coanalytic. Notice that the complexity of the graph is as low as possible by Corollary \ref{corollary_closed_graph}. Furthermore, this result yields a consistent example of a coanalytic non-Effros group (see Corollary \ref{corollary_counterexample_V=L}), thus showing that Theorem \ref{theorem_van_mill} is sharp.

We begin with the ``coding lemma'' needed in the proof of Theorem \ref{theorem_counterexample_V=L}. Throughout this section, we will make the same conventions described in the paragraph that precedes Theorem \ref{theorem_zoltan_multivariable}. Similarly, we will identify $2^\omega\times 2^\omega$ and $2^\omega$ for the purposes of Turing reduction.
\begin{lemma}\label{lemma_equicofinal}
Let $a,z\in 2^\omega$, and let $\{(z_n,w_n):n\in\omega\}$ be a countable subset of $2^\omega\times 2^\omega$. Then there exists $w\in 2^\omega$ satisfying the following properties:
\begin{itemize}
\item\label{lemma_equicofinal_equivalent} $\big((z_m+z,w_m+w):m\in\omega\big)\equivT (z_n+z,w_n+w)$ for every $n\in\omega$,
\item\label{lemma_equicofinal_high} $a\leqT (z_n+z,w_n+w)$ for every $n\in\omega$.
\end{itemize}
\end{lemma}
\begin{proof}
Fix infinite $\Omega_{(n,\varepsilon)}\subseteq\omega$ for $(n,\varepsilon)\in\omega\times 3$ such that the following conditions hold:
\begin{itemize}
\item $\Omega_{(n,\varepsilon)}\cap\Omega_{(m,\delta)}=\varnothing$ whenever $(n,\varepsilon)\neq (m,\delta)$,
\item $\bigcup\{\Omega_{(n,\varepsilon)}:(n,\varepsilon)\in\omega\times 3\}=\omega$,
\item $\{(n,\varepsilon,k)\in\omega\times 3\times\omega:k\in\Omega_{(n,\varepsilon)}\}$ is recursive.
\end{itemize}
Given $x_n\in 2^\omega$ for $n\in\omega$, let $\oplus(x_n:n\in\omega)$ denote a uniform recursive way of coding the $x_n$ into a single element of $2^\omega$. Given $\varepsilon\in 3$, set
$$
\left.
\begin{array}{lcl}
& & u_\varepsilon=\left\{
\begin{array}{ll}
a & \text{if }\varepsilon=0,\\
\oplus(z+z_n:n\in\omega) & \text{if }\varepsilon=1,\\
\oplus(w_n:n\in\omega) & \text{if }\varepsilon=2.
\end{array}
\right.
\end{array}
\right.
$$
Also fix bijections $\pi_{(n,\varepsilon)}:\omega\longrightarrow\Omega_{(n,\varepsilon)}$ for $(n,\varepsilon)\in\omega\times 3$ such that the function $(n,\varepsilon,k)\longmapsto\pi_{(n,\varepsilon)}(k)$ is recursive. Define
$$
w(k)=w_n(k)+u_\varepsilon\left(\pi_{(n,\varepsilon)}^{-1}(k)\right)
$$
for every $k\in\Omega_{(n,\varepsilon)}$. In order to see that $w$ has the desired properties, pick $n\in\omega$. The fundamental observation is that from $(w_n+w)\upharpoonright\Omega_{(n,\varepsilon)}$ it is possible to recover $u_\varepsilon$ in a recursive way. This shows that $a\leqT w_n+w$ and $\oplus(z+z_m:m\in\omega)\leqT w_n+w$.

It remains to see that $\oplus(w_m+w:m\in\omega)\leqT w_n+w$, which is only slightly more difficult. As above, one sees that $\oplus(w_m:m\in\omega)\leqT w_n+w$, and in particular $w_n\leqT w_n+w$. Now it is clear that $w\leqT w_n+w$, hence $\oplus(w_m+w:m\in\omega)\leqT w_n+w$, as desired.
\end{proof}

\begin{theorem}\label{theorem_counterexample_V=L}
Assume $\mathsf{V=L}$. Then there exists a discontinuous homomorphism $\varphi:2^\omega\longrightarrow 2^\omega$ with coanalytic graph.	
\end{theorem}
\begin{proof}
Set $Z=2^\omega\times 2^\omega$ and $M=Z^\omega$. We begin by constructing a countable dense subgroup $G_0$ of $Z$ that is the graph of a homomorphism $\varphi_0$ between (countable) subgroups of $2^\omega$. First define $\mathbf{e}_n\in 2^\omega$ for $n\in\omega$ by setting
$$
\left.
\begin{array}{lcl}
& & \mathbf{e}_n(m)=\left\{
\begin{array}{ll}
1 & \text{if }m=n,\\
0 & \text{if }m\neq n
\end{array}
\right.
\end{array}
\right.
$$
for $m\in\omega$. Let $E$ denote the subgroup of $2^\omega$ generated by $\{\mathbf{e}_n:n\in\omega\}$. Fix $\pi:\omega\longrightarrow\omega$ such that $\pi^{-1}(n)$ is infinite for every $n\in\omega$. Let $\varphi_0:E\longrightarrow 2^\omega$ be the unique homomorphism such that $\varphi_0(\mathbf{e}_n)=\mathbf{e}_{\pi(n)}$ for each $n$. Set $G_0=\Gr(\varphi_0)$, and let $G_0=\{g_n:n\in\omega\}$ be an enumeration.

Our plan is to construct a set $X\subseteq M$ such that
$$
G=G_0\cup\{x(n):x\in X\text{ and }n\in\omega\}
$$
is the graph of a homomorphism $\varphi:2^\omega\longrightarrow 2^\omega$. Since $G_0\subseteq G$ and $G_0$ is dense in $Z$, it is clear that $\varphi$ will be discontinuous. On the other hand, using Theorem \ref{theorem_zoltan_multivariable}, we will make sure that $G$ is coanalytic.

Given $A\in M^\eta$, where $\eta\leq\omega$, we will define $A^\ast:\omega\longrightarrow Z$. Fix a bijection
$$
\tau:\omega\longrightarrow (\omega\times\{0\})\cup (\eta\times\{1\}).
$$
Also, in the case $\eta>0$, fix $\sigma_0:\omega\longrightarrow\eta$ and $\sigma_1:\omega\longrightarrow\omega$ such that the function $(\sigma_0,\sigma_1):\omega\longrightarrow\eta\times\omega$ defined by setting $(\sigma_0,\sigma_1)(m)=(\sigma_0(m),\sigma_1(m))$ for $m\in\omega$ is a bijection (these functions are not needed in the case $\eta=0$). Finally, define $A^\ast$ by setting
$$
\left.
\begin{array}{lcl}
& & A^\ast(n)=\left\{
\begin{array}{ll}
g_m & \text{if }\tau(n)=(m,0),\\
A(\sigma_0(m))(\sigma_1(m)) & \text{if }\tau(n)=(m,1)
\end{array}
\right.
\end{array}
\right.
$$
for $n\in\omega$. The purpose of $A^\ast$ is simply to transform $A$ (which is a sequence of sequences of pairs) into a sequence of pairs that enumerates all those encoded by $A$, plus those in $G_0$.

The set of ``conditions'' to be satisfied will be $B=2^\omega$. More precisely, we will make sure that each $p\in B$ will be added to the domain of our homomorphism at some stage. We will denote by $\pi_i:Z\longrightarrow 2^\omega$ for $i=0,1$ the natural projections, so that $u=(\pi_0(u),\pi_1(u))$ for every $u\in Z$.

Define $F\subseteq M^{\leq\omega}\times B\times M$ by declaring $(A,p,x)\in F$ exactly when one of the following conditions holds:
\begin{itemize}
\item $\exists n\in\omega\,\big(p=\pi_0(A^\ast(n))\big)$ and\\$\forall m,n\in\omega\,\big(\pi_0(x(m)-A^\ast(m))=\pi_0(x(n)-A^\ast(n))\big)$ and\\$\forall m,n\in\omega\,\big(\pi_1(x(m)-A^\ast(m))=\pi_1(x(n)-A^\ast(n))\big)$ and\\$\forall m,n\in\omega\,\big(\pi_0(x(m)-A^\ast(m))\neq\pi_0(A^\ast(n))\big)$, or
\item $\forall n\in\omega\,\big(p\neq\pi_0(A^\ast(n))\big)$ and \\$\forall n\in\omega\,\big(\pi_0(x(n)-A^\ast(n))=p\big)$ and\\$\forall m,n\in\omega\,\big(\pi_1(x(m)-A^\ast(m))=\pi_1(x(n)-A^\ast(n))\big)$.
\end{itemize}
The first case is the one in which $p$ already belongs to the domain of our intended partial homomorphism. In this case, we ask that there exist $z,w\in~2^\omega$ such that $z$ is outside of this domain and $x$ is an enumeration of $\{(z_n+z,w_n+w):n\in\omega\}$, where $\{(z_n,w_n):n\in\omega\}$ is an enumeration of $\{A^\ast(n):n\in\omega\}$. The second case is almost the same: the only difference is that we require $z=p$. In either case, we will have made sure that $p$ belongs to the domain of the new partial homomorphism and that $x$ enumerates the elements of $Z$ that have been added to its graph.

It is clear that $F$ is coanalytic (in fact, it is Borel). Given any $(A,p)\in M^{\leq\omega}\times B$, Lemma \ref{lemma_equicofinal} shows that $F_{(A,p)}$ is equicofinal in the Turing degrees. Therefore, by Theorem \ref{theorem_zoltan_multivariable}, we can fix $X\subseteq M$ such that the following conditions are satisfied:
\begin{itemize}
\item $X$ is compatible with $F$,
\item $\{x(n):x\in X\text{ and }n\in\omega\}$ is coanalytic.
\end{itemize}

Let $A_\alpha$, $p_\alpha$ and $x_\alpha$ for $\alpha<\omega_1$ be the enumerations given by the definition of compatibility. Given $0<\alpha\leq\omega_1$, set
$$
G_\alpha=G_0\cup\{x_\beta(n):\beta<\alpha\text{ and }n\in\omega\}.
$$
We claim that each $G_\alpha$ is the graph of a homomorphism $\varphi_\alpha:\pi_0[G_\alpha]\longrightarrow 2^\omega$.

The case $\alpha=0$ has been discussed at the beginning of this proof, and the limit case is straightforward. Now assume that the claim holds for some $\alpha<\omega_1$. Since $x_\alpha\in F_{(A_\alpha,p_\alpha)}$ by compatibility, there exists $(z,w)\in Z$ such that $z\notin\pi_0[G_\alpha]$ and 
$$
G_{\alpha+1}=G_\alpha\cup\{x_\alpha(n):n\in\omega\}=G_\alpha\cup\{(z_n+z,w_n+w):n\in\omega\},
$$
where $\{(z_n,w_n):n\in\omega\}$ is an enumeration of $\{A_\alpha^\ast(n):n\in\omega\}$. On the other hand, it is clear that $\{A_\alpha^\ast(n):n\in\omega\}$ is an enumeration of $G_\alpha$, as $\{A_\alpha(n):n\in\eta\}$ is an enumeration of $\{x_\beta:\beta<\alpha\}$ by compatibility, where $A_\alpha\in M^\eta$. It follows that the right-hand side is the subgroup of $Z$ generated by $G_\alpha\cup\{(z,w)\}$, hence the claim holds for $\alpha+1$. Furthermore, as discussed above, we will have $p_\alpha\in\pi_0[G_{\alpha+1}]$ for every $\alpha<\omega_1$. In conclusion, the group $G=G_{\omega_1}$ will be the graph of the desired homomorphism $\varphi:2^\omega\longrightarrow 2^\omega$.
\end{proof}

\begin{corollary}\label{corollary_counterexample_V=L}
Assume $\mathsf{V=L}$. Then there exists a coanalytic non-Effros group.
\end{corollary}
\begin{proof}
Combine Theorems \ref{theorem_counterexample_V=L} and \ref{theorem_closed_graph}.
\end{proof}

\subsection*{Acknowledgements}
The author acknowledges the support of the FWF grant P 30823. He also thanks the anonymous referee for a careful reading of the paper and several thoughtful suggestions.


\begin{thebibliography}{99}

\bibitem[An]{ancel}\textsc{F. D. Ancel.} An alternative proof and applications of a theorem of E. G. Effros. \emph{Michigan Math. J.} \textbf{34:1} (1987), 39--55.

\bibitem[CMM]{carroy_medini_muller}\textsc{R. Carroy, A. Medini, S. M\"uller.} Constructing Wadge classes. To appear in \emph{Bull. Symb. Log.} Available on \verb"arxiv.org".

\bibitem[CM]{charatonik_mackowiak}\textsc{J. J. Charatonik, T. Ma\'{c}owiak.} Around Effros' theorem. \emph{Trans. Amer. Math. Soc.} \textbf{298:2} (1986), 579--602.

\bibitem[Ef]{effros}\textsc{E. G. Effros.} Transformation groups and $C^\ast$-algebras. \emph{Ann. of Math. (2)} \textbf{81} (1965), 38--55.

\bibitem[EKM]{erdos_kunen_mauldin}\textsc{P. Erd\H{o}s, K. Kunen, R. D. Mauldin.} Some additive properties of sets of real numbers. \emph{Fund. Math.} \textbf{113:3} (1981), 187--199.

\bibitem[FST]{fischer_schrittesser_tornquist}\textsc{V. Fischer, D. Schrittesser, A. T\"{o}rnquist.} A co-analytic Cohen-indestructible maximal cofinitary group. \emph{J. Symb. Log.} \textbf{82:2} (2017), 629--647.

\bibitem[Gl]{glimm}\textsc{J. Glimm.} Locally compact transformation groups. \emph{Trans. Amer. Math. Soc.} \textbf{101} (1961), 124--138.

\bibitem[Je]{jech}\textsc{T. Jech.} \emph{Set theory.} The third millennium edition, revised and expanded. Springer Monographs in Mathematics. Springer-Verlag, Berlin, 2003. xiv+769 pp.

\bibitem[Ka]{kastermans}\textsc{B. Kastermans.} The complexity of maximal cofinitary groups. \emph{Proc. Amer. Math. Soc.} \textbf{137:1} (2009), 307--316.

\bibitem[Ke1]{kechris_1994}\textsc{A. S. Kechris.} Topology and descriptive set theory. \emph{Topology Appl.} \textbf{58:3} (1994), 195--222.

\bibitem[Ke2]{kechris_1995}\textsc{A. S. Kechris.} \emph{Classical descriptive set theory.} Graduate Texts in Mathematics, 156. Springer-Verlag, New York, 1995. xviii+402 pp.

\bibitem[Me]{medini}\textsc{A. Medini.} On Borel semifilters. \emph{Topology Proc.} \textbf{53} (2019), 97--122.

\bibitem[MZ]{medini_zdomskyy}\textsc{A. Medini, L. Zdomskyy.} Between Polish and completely Baire. \emph{Arch. Math. Logic.} \textbf{54:1-2} (2015), 231--245.

\bibitem[vM1]{van_mill_2004}\textsc{J. van Mill.} A note on the Effros theorem. \emph{Amer. Math. Monthly} \textbf{111:9} (2004), 801--806.

\bibitem[vM2]{van_mill_2008}\textsc{J. van Mill.} Homogeneous spaces and transitive actions by Polish groups. \emph{Israel J. Math.} \textbf{165} (2008), 133--159.

\bibitem[Mi1]{miller_1989}\textsc{A. W. Miller.} Infinite combinatorics and definability. \emph{Ann. Pure Appl. Logic} \textbf{41:2} (1989), 179--203.

\bibitem[Mi2]{miller_1995}\textsc{A. W. Miller.} \emph{Descriptive set theory and forcing.} Lecture Notes in Logic, 4. Springer-Verlag, Berlin, 1995. ii+130 pp.

\bibitem[Mo]{moschovakis}\textsc{Y. N. Moschovakis.} \emph{Descriptive set theory.} Second edition. Mathematical Surveys and Monographs, 155. American Mathematical Society, Providence, RI, 2009. xiv+502 pp.

\bibitem[Od]{odifreddi}\textsc{P. Odifreddi.} \emph{Classical recursion theory. The theory of functions and sets of natural numbers.} Studies in Logic and the Foundations of Mathematics, 125. North-Holland Publishing Co., Amsterdam, 1989. xviii+668 pp.

\bibitem[Os1]{ostaszewski_2013}\textsc{A. J. Ostaszewski.} Almost completeness and the Effros open mapping principle in normed groups. \emph{Canad. Math. Bull.} \textbf{58:2} (2015), 334--349.

\bibitem[Os2]{ostaszewski_2015}\textsc{A. J. Ostaszewski.} Effros, Baire, Steinhaus and non-separability. \emph{Topology Appl.} \textbf{195} (2015), 265--274.

\bibitem[Ro]{rosendal}\textsc{C. Rosendal.} Automatic continuity of group homomorphisms. \emph{Bull. Symbolic Logic} \textbf{15:2} (2009), 184--214.

\bibitem[Vi]{vidnyanszky}\textsc{Z. Vidny\'anszky.} Transfinite inductions producing coanalytic sets. \emph{Fund. Math.} \textbf{224:2} (2014), 155--174.

\end{thebibliography}
\end{document}